\documentclass[11pt]{article}
\usepackage{arxiv}

\usepackage{pgfplots,tikz,url,todonotes}

\usepackage{fullpage}
\usepackage[hidelinks]{hyperref}

\usepackage{amsmath,amssymb,amsthm}
\usepackage[capitalize]{cleveref}
\usepackage{mathtools}
\usepackage{graphicx}
\usepackage{enumitem}

\pgfplotsset{compat=1.10}

\usepackage{palatino}
\usepackage{mathpazo}
\usepackage{fullpage}
\usepackage[margin=1cm]{caption}

\usepackage{thmtools}
\usepackage{thm-restate}

\newcommand{\ignore}[1]{}
\newcommand{\N}{\mathbb{N}}

\newcommand{\R}{\mathbb{R}}

\newcommand{\mK}{\operatorname{\mathcal K}}

\renewcommand{\sup}{\mathrm{sup}}
\renewcommand{\min}{\mathrm{min}}

\renewcommand{\max}{\mathrm{max}}
\renewcommand{\epsilon}{\varepsilon}
\newcommand{\eps}{\epsilon}

\newcommand{\SoS}{\Sigma[x]}

\def\01{\{0,1\}}

\newtheorem{defin}{Definition} %[section]

\newtheorem{proposition}[defin]{Proposition}
\newtheorem{theorem}{Theorem}
\newtheorem*{theorem*}{Theorem}
\newtheorem{remark}[defin]{Remark}

\newtheorem{lemma}[defin]{Lemma}

\newtheorem*{claim*}{Claim}

\newtheorem*{conjecture*}{Conjecture}
\theoremstyle{definition}

\newcommand{\mon}{{1,\mathrm{mon}}}
\newcommand{\cheb}{{1,\mathrm{cheb}}}

\newcommand{\x}{\mathbf{x}}
\newcommand{\y}{\mathbf{y}}

\newcommand{\Chebyconst}{\mathfrak{c}}

\date{\today}
\begin{document}

\title{Squared polynomial approximation kernels for the hypercube: improved error bounds and implications for Lasserre hierarchies}
\author{Sander Gribling\thanks{S.J.Gribling@tilburguniversity.edu} \and Etienne de Klerk\thanks{E.deKlerk@tilburguniversity.edu} \and Juan C.\ Vera\thanks{J.C.VeraLizcano@tilburguniversity.edu} }
\date{Department of Econometrics and Operations Research, Tilburg University, The Netherlands}

\maketitle

\begin{abstract}
We propose a new family of polynomial approximation kernels for approximating nonnegative polynomials on the hypercube  $[-1,1]^n$. Our Kernels produce polynomial sums-of-squares 
of degree $r$, achieving an $O(\log^3 r/r^2)$  error in the $\ell_1$-norm of the coefficients.
This improves on the known error bound $O(1/r)$ from the literature.
As a corollary, we obtain an improved convergence rate for the Lasserre hierarchy for polynomial optimization on the hypercube, again improving a known rate by Baldi and Slot from  $O(1/r)$ to $O(\log^3 r/r^2)$.

\end{abstract}

\section{Introduction}
We consider the set of multivariate polynomials of degree at most $d$ in the variables $\x = (x_1,\ldots,x_n)$, denoted by $\mathbb{R}[\x]_d$, that are also nonnegative on the hypercube
$[-1,1]^n$.
A natural question is how well such a polynomial may be approximated (in a suitable norm) by a sum-of-squares of polynomials with total degree $r$; we will denote the cone of such polynomials by $\Sigma[\x]_r$.
Thus we are interested in projections of polynomials onto $\Sigma[\x]_r$, using a suitable norm.
To introduce one widely-studied  norm, we will write a polynomial $f \in \mathbb{R}[\x]_d$ in the standard monomial basis
as
\[
f(\x) = \sum_{\alpha \in \mathbb{N}^n_d} f_\alpha \x^\alpha,
\]
where
$\mathbb{N}^n_d = \{\alpha \in (\mathbb{N}_0)^n \; : \;  \sum_{i=1}^n \alpha_i \le d\}$, and
$\x^\alpha := x_1^{\alpha_1}\cdots x_n^{\alpha_n}$. We may then define the $1$-norm of $f$ in the monomial basis as
\[
\|f\|_\mon := \sum_{\alpha \in \mathbb{N}^n_d} |f_\alpha|.
\]

Thus, given $f \in \mathbb{R}[\x]_d$  such that $f$ is nonnegative on $[-1,1]^n$,  we consider the projection onto
$\Sigma[\x]_r$, namely
\[
\mbox{Proj}^{\|\cdot\|_\mon}_{\Sigma_r[\x]}(f) := \arg\min_{p \in \Sigma_r[\x]} \|f-p\|_\mon,
\]
and ask how the error $\left\|\mbox{Proj}^{\|\cdot\|_\mon}_{\Sigma_r[\x]}(f)-f\right\|_\mon$ depends on $r$. It is known that the cone of sums of squared polynomials is dense in the cone of polynomials that are nonnegative on $[-1,1]^n$ in the $\|\cdot\|_\mon$ norm; see, e.g., Berg \cite[Theorem 5, p. 117]{Berg}.
Lasserre \cite{doi:10.1137/04061413X} (see also \cite{Lasserre_Netzer}) gave a constructive proof of this density result, and later also gave the following characterization of the projection in the unpublished preprint \cite{Lasserre_arXiv}.
\begin{theorem}[Lasserre \cite{Lasserre_arXiv}]
Assume $f \in \mathbb{R}[\x]_d$ and $r \ge d$ and $r$ is even.
Then there exist nonnegative
$\lambda^*_0, \ldots, \lambda_n^*$, such that, for the $\|\cdot\|_\mon$ norm,
\[
\mbox{\rm Proj}^{\|\cdot\|_\mon}_{\Sigma_r[\x]}(f)(\mathbf{x}) =  f(\mathbf{x}) + \lambda_0^* + \sum_{i=1}^n \lambda^*_i x_i^{r},
\]
i.e.\ $\|\mbox{\rm Proj}^{\|\cdot\|_\mon}_{\Sigma_r[\x]}(f)-f\|_\mon = \sum_{i=0}^n \lambda^*_i$.
The values $\lambda^*_0, \ldots, \lambda_n^*$ are given by the optimal solution of the semidefinite programming (SDP) problem
\[
 \min_{\lambda \ge 0} \left\{ \sum_{i=0}^n \lambda_i \; : \; \x \mapsto f(\mathbf{x}) +
 \lambda_0 + \sum_{i=1}^n \lambda_i x_i^{r} \in \Sigma[\mathbf{x}]_{r}\right\}.
\]
\end{theorem}
By the aforementioned density result, one has $\left\|\mbox{\rm Proj}^{\|\cdot\|_\mon}_{\Sigma_r[\x]}(f)-f\right\|_\mon \rightarrow 0$ as $ r \rightarrow \infty$. However, there are no known upper bounds on $\left\|\mbox{\rm Proj}^{\left\|\cdot\right\|_\mon}_{\Sigma_r[\x]}(f)-f\right\|_\mon$ in terms of $r$.

\subsection*{Contributions of this paper}
In this paper, we will study the closely related projection that uses the $1$-norm of the coefficients in the {\em Chebyshev basis}, denoted by $\|\cdot\|_\cheb$; see Section \ref{sec:chebyshev polynomials}. Although the resulting norm is an upper bound on the $1$-norm in the monomial basis (see Section \ref{sec:norms of polynomials}), it facilitates the analysis to work with an orthogonal polynomial basis. In particular, working with the Chebyshev basis also allows one to derive error bounds on 
$\left\|\mbox{\rm Proj}^{\left\|\cdot\right\|_\cheb}_{\Sigma_r[\x]}(f)-f\right\|_\cheb$ in terms of $r$.

The first such result  was an $O(1/r)$ bound, implicit in the analysis of Baldi and Slot \cite{doi:10.1137/23M1555430}; see also Gribling et al.\ \cite{GdKV26} for a different proof of the same result.
Our main result in this paper is to improve this to
\[
\left\|\mbox{\rm Proj}^{\|\cdot\|_\mon}_{\Sigma_r[\x]}(f)-f\right\|_\cheb = O\left(\frac{\log^3 r}{r^2}\right);
\]
see Theorem \ref{thm:main} below for a precise statement. The main idea of our proof is to use the polynomial kernel method; see e.g.\ \cite{RevModPhys.78.275}. This method  has been used in several recent papers dealing with sum-of-squares approximation, e.g.\ \cite{Laurent2021AnEV,10.1007/s10107-020-01537-7,GdKV26}. In fact, the improvement we obtain is possible due to the construction of a better kernel than that used in \cite{GdKV26}.

Results of this type have immediate consequences for the convergence rate of the Lasserre hierarchy \cite{doi:10.1137/S1052623400366802} of lower bounds for polynomial optimization in the hypercube, as shown in \cite{GdKV26}. In particular, we  improve the best known result on the convergence rate, due to Baldi and Slot \cite{doi:10.1137/23M1555430}, from $O(1/r)$ to $O\left(\log^3 r/r^2\right)$; see Theorem~\ref{thm:new error bound lasserre lower bounds}.

In addition, our new kernel construction also allows us to analyze  a different hierarchy by Lasserre \cite{Las11} of {\em upper} bounds
for polynomial optimization on the hypercube. In particular, we give a new and constructive proof of the result of the $O(1/r^2)$ convergence rate by De Klerk and Laurent \cite{DKL MOR2}; see Proposition \ref{prop:constructive proof}. Our new proof makes explicit how the rate of convergence depends on the parameters $n$, $d$ and $\|f\|_\cheb$, where $f$ is the objective function of degree $d$.
In fact, in presenting all our results, we will mostly avoid the big-O notation, in order to make the dependence on parameters other than $r$ clear as well.

\subsection*{Outline of this paper}
Since we will work with the Chebyshev basis, we will first review some properties of these polynomials in Section \ref{sec:chebyshev polynomials}, as well as the relations between the related polynomial norms in Section \ref{sec:norms of polynomials}. To conclude the preliminary material, we review how certain rational functions may be approximated by sums-of-squares of polynomials through truncated geometric series in Section \ref{sec:divide}.

Section \ref{sec:squared kernels} contains the kernel construction that we need for the main result.
The basic idea is to use squared univariate kernels; see Section \ref{sec:Analysis of squared univariate kernels}. This construction naturally leads to rational approximations, and we therefore show in Section \ref{sec:A sum-of-squares approximation} how to approximate these rational functions by sums of squared polynomials using truncated geometric series.
In Section \ref{sec:phi} we show how to leverage knowledge of known univariate kernels in our construction;
we construct a family of kernels, some of which have desirable properties, the squared Fej\'er kernel being one of them, as shown in Section \ref{sec:Fejer}. Finally, we review a well-known technique to extend the analysis to the multivariate case by multiplying univariate kernels; see Section \ref{sec:multivariate extension}.
Section \ref{sec:Main result and its consequences} contains our main result and its implications.
We first state our main result in Section \ref{sec:SOS approximations of nonnegative polynomials on the hypercube}, followed by its implications for the Lasserre hierarchies of lower and upper bounds in Sections \ref{sec:Improved rate of convergence for the Lasserre hierarchy of lower bounds} and \ref{sec:A new convergence rate proof for the Lasserre spectral hierarchy of upper bounds} respectively.
We conclude the paper in Section \ref{sec:conclusion} with a discussion of tightness of some error bounds and possible extensions of our work.

\section{Preliminaries}
\subsection{Chebyshev polynomials}
\label{sec:chebyshev polynomials}
Here we review some properties of Chebyshev polynomials for later use; more details and proofs may be found in the book by Rivlin \cite{rivlin2020chebyshev}.

We will use the  univariate   Chebyshev polynomials (of the first kind),
 defined by:
\begin{equation}\label{eqTUn}
T_k(x)= \cos(k\arccos (x)),\ \ \text{ for } \ x\in [-1,1],\ k = 0,1,\ldots.
\end{equation}
They satisfy the following three-terms recurrence relationships:
\begin{equation}\label{eqTnrec}
T_0(x)=1,\ %T_{-1}(x)=
T_1(x)=x, \ T_{k+1}(x)= 2xT_k(x)-T_{k-1}(x) \ \ \text{ for } \ k\ge 1.
\end{equation}
The Chebyshev polynomials are orthogonal with respect to the weight function
  $w(x) :=\frac{1}{ \pi\sqrt{1-x^2}}$, and one has the associated Chebyshev probability measure $d\mu(x) = w(x)dx$ on $[-1,1]$. In particular
$$
\langle T_i,T_j\rangle := \int_{-1}^1 T_i(x)T_j(x) d\mu(x)=
\left\{
\begin{array}{rl}
0 & \text{ if } i\ne j, \\
1 & \text{ if } i=j=0,\\
\frac12 & \text{ if } i=j\ge 1.
\end{array}
\right.$$
We may therefore define orthonormal Chebyshev polynomials w.r.t.\ $\mu$ via
\[
\widehat T_0(x)= T_0(x) =1,\qquad \widehat T_k(x)=\sqrt{2}\,T_k(x) \;\; (k\ge 1)
\]
so that, for \(k\ge 2\),
\[
\widehat T_{k+1}(x)+\widehat T_{k-1}(x)=2x\,\widehat T_k(x).
\]

For $\alpha \in (\mathbb{N}_0)^n$ the multivariate Chebyshev polynomial of the first kind is defined as
	\[
	T_\alpha(\textbf{x}) = \prod_{i=1}^n T_{\alpha_i}(x_i),
	\]
with  the normalized polynomial $\widehat T_\alpha$ defined analogously.

\subsection{Norms of polynomials}
\label{sec:norms of polynomials}
We define the sup-norm $\|\cdot\|_\infty$ of $f \in \R[\x]_d$ on $[-1,1]^n$ as
 \[
 \|f\|_\infty = \max_{\x \in [-1,1]^n} |f(\x)|.
 \]
If we write a multivariate, degree $d$ polynomial $f$ in the Chebyshev basis as, $f(\x) = \sum_{\alpha \in \mathbb{N}^n_d} f_\alpha T_\alpha(\x)$, then
the corresponding $\ell_1$-norm becomes $\|f\|_\cheb := \sum_{\alpha \in \mathbb{N}^n_d} |f_\alpha|$.

The following (equivalence) relations hold between the norms we have introduced:
\begin{equation}\label{norm equivalence relations}
  \|f\|_{\infty} \le \|f\|_\cheb \le \left(2^n \binom{n+d}{d}\right)^{ \frac{1}{2}}\|f\|_{\infty},
\end{equation}
see e.g.\  \cite[Lemma 2.3]{GdKV26},\footnote{The proof in \cite[Lemma 2.3]{GdKV26} is only for the univariate case ($n=1$), but the proof for general $n$ proceeds in exactly the same way.} 
and
\begin{equation}\label{norm equivalence relations2}
 \|f\|_\cheb \le \|f\|_\mon \le \left(\frac{5}{2} \right)^d \|f\|_\cheb,
\end{equation}
where the first inequality in \eqref{norm equivalence relations2} is proved, e.g. in \cite[Lemma 2.1]{De_Klerk_Vera_2024}, and the second inequality may easily be shown using induction in $d$ and the recursive relation
\eqref{eqTnrec}, as well as the sub-multiplicativity of the $1$-norm.

\subsection{Polynomial approximation of the reciprocal of a positive polynomial} \label{sec:divide}

Let $q(x)$ be a polynomial such that $0<c<q(x)\leq C$ for $x \in [-1,1]$. We construct a sum-of-squares polynomial approximation of $1/q(x)$. We do so using a geometric series and we first recall two useful identities.

For $t \neq 1$ we have
\begin{equation} \label{eq:geosum}
\sum_{m=0}^N t^m = \frac{1-t^{N+1}}{1-t}.
\end{equation}
For even $N$ we have
\[
\sum_{m=0}^{N} t^m = \frac 12\left(1+ \sum_{m=0}^{N/2 - 1}t^{2m}(1+t)^2 + t^{N}\right).
\]
In particular, for even $N$ and any polynomial $p$, we have that
\begin{equation} \label{eq:geoSOS}
\sum_{m=0}^{N} p(x)^m \in \SoS.
\end{equation}

\begin{lemma} \label{lem:ratio}
For each \(N\ge 0\), define
\[
p_{N}(x):=\frac1{C}\sum_{m=0}^{N}
\left(1-\frac{q(x)}{C}
\right)^m,
\]
Then
\[
\left\|
\frac1{q}-p_{N}
\right\|_\infty
\le
\frac{1}{c}\left(1-\frac{c}{C}\right)^{N+1}.
\]
The polynomial $p_N$ has the following properties: $\deg(p_N) = N \deg(q)$,
\begin{equation} \label{eq:geocheb}
\|p_N\|_{\cheb} \leq \frac{1}{C} \sum_{m=0}^N \left\|1-\frac{q}{C}\right\|_{\cheb}^m,
\end{equation} and, if $N$ is even, we have $p_N \in \SoS$.
\end{lemma}

\begin{proof}
By assumption, for any $x \in [-1,1]$ we have $0<c\leq q(x)\leq C$. Hence $0 \leq 1-\frac{q(x)}{C} \leq 1-\frac{c}{C}<1$.
Therefore the geometric series converges uniformly on \([-1,1]\), and thus
\[
\frac1{C}\sum_{m=0}^{\infty}
\left(1-\frac{q(x)}{C}
\right)^m
 = 
\frac{1}{C} \frac{1}{1-\left(1-\frac{q(x)}{C}\right)}
=
 \frac1{q(x)}.
\]

That is, \(p_{N}\to 1/q\) uniformly.
Finally, we bound the error using \eqref{eq:geosum}: for $x \in [-1,1]$ we have
\[
\left|
\frac1{q(x)}-p_N(x)
\right|
\le
\frac{\left(1-\frac{q(x)}{C}
\right)^{N+1}}{q(x)} \leq \frac{\left(1-\frac{c}{C}\right)^{N+1}}{c}.
\]
Using the triangle inequality and sub-multiplicativity of the $1$-norm we moreover have
\begin{equation*}% \label{eq:geocheb}
\|p_N\|_{\cheb} \leq \frac{1}{C} \sum_{m=0}^N \left\|1-\frac{q}{C}\right\|_{\cheb}^m.
\end{equation*}
For $N$ even, we have $p_N \in \SoS$ using \cref{eq:geoSOS}.
\end{proof}

In what follows, we will apply Lemma \ref{lem:ratio} to polynomials $q$ with a special structure, as introduced in the next lemma.
\begin{lemma} \label{lem:1qbound}
Let $a \in [0,1]^r$ and define
\[
q(x)=1 +2\sum_{j=1}^r (1-a_j)^2 T_j(x)^2.
\]
Then for $x \in [-1,1]$ we have $q(x) \leq 1 +2\sum_{j=1}^r (1-a_j)^2 =:C$ and we have
\[
\left\|1-\frac{q}{C}\right\|_{\cheb} =  \frac{C-1}{C} \leq 1.
\]
\end{lemma}
\begin{proof}
Since $T_j(x)^2 \leq 1$ for all $x \in [-1,1]$ we have $q(x) \leq 1 +2\sum_{j=1}^r (1-a_j)^2$. Let $C:= 1 +2\sum_{j=1}^r (1-a_j)^2$.
Then we have
\begin{align*}
\left\|1-\frac{q}{C}\right\|_{\cheb} &= \frac1{C}\left\|C -(1+2\sum_{j=1}^r (1-a_j)^2 T_j^2)\right\|_{\cheb} \\
&= \frac1{C}\left(\left|1 +2\sum_{j=1}^r (1-a_j)^2-(1+\sum_{j=1}^r (1-a_j)^2) \right| + \sum_{j=1}^r (1-a_j)^2\right) \\
&= \frac{2\sum_{j=1}^r (1-a_j)^2}{C} = \frac{C-1}{C} \leq 1.
\end{align*}
\end{proof}

\section{A new polynomial kernel on the hypercube}
\label{sec:squared kernels}
In this section we will do a systematic analysis of a new class of approximation kernels. 
The basic idea will be to  take a univariate kernel that is a good approximation to the identity, for instance %square of classical univariate kernels like 
the Fej\'er kernel, and square it. We then analyze the squared kernel and show that it is, up to normalization and under certain conditions, close to the identity. We refer to %; see the review 
\cite{RevModPhys.78.275} for an overview of classical univariate kernels on $[-1,1]$.

\subsection{Analysis of squared univariate kernels}
\label{sec:Analysis of squared univariate kernels}
Let $a \in [0,1]^r$ and set
\[
S_r^{(a)}(x,y):=1+2\sum_{j=1}^r (1-a_j)\,T_j(x)T_j(y),
\]
and define the positive kernel
\[
K_r^{(a)}(x,y):=\bigl(S_r^{(a)}(x,y)\bigr)^2.
\]
Let \(\mK_r^{(a)}\) be the associated convolution operator
\[
\mK_r^{(a)}(f)(x):=\int_{-1}^1 f(y)\,K_r^{(a)}(x,y)\,d\mu(y),
\]
and set
\[
M_r^{(a)}(x):=\mK_r^{(a)}(1)(x)
=
1+2\sum_{j=1}^r (1-a_j)^2\,T_j(x)^2.
\]
We extend the sequence \((a_j)\) by \(a_m:=1\) for all \(m>r\) and set $a_0 := 0$.

\begin{lemma}\label{lem:SymSa} For $0 \leq k \leq r$ we have
\begin{align*}
 T_k(y)S_r^{(a)}(x,y)-T_k(x)S_r^{(a)}(x,y)
&=
a_k\bigl(T_k(y)-T_k(x)\bigr) \\
&\qquad
+\sum_{j=1}^r (a_{j+k}-a_j)
\Bigl(T_j(x)T_{j+k}(y)-T_{j+k}(x)T_j(y)\Bigr)\\
&\qquad
+\frac12\sum_{j=1}^{k-1}(a_{k-j}-a_j)
\Bigl(T_j(x)T_{k-j}(y)-T_{k-j}(x)T_j(y)\Bigr).
\end{align*}
\end{lemma}
\begin{proof}
Using the product formula
\(
2T_jT_k=T_{j+k}+T_{|j-k|},
\)
we obtain
\[
T_k(y)S_r^{(a)}(x,y)
=
T_k(y)
+\sum_{j=1}^r (1-a_j)\,T_j(x)\bigl(T_{j+k}(y)+T_{|j-k|}(y)\bigr),
\]
and similarly
\[
T_k(x)S_r^{(a)}(x,y)
=
T_k(x)
+\sum_{j=1}^r (1-a_j)\,T_j(y)\bigl(T_{j+k}(x)+T_{|j-k|}(x)\bigr).
\]
Subtracting, we get
\begin{align}
\notag &T_k(y)S_r^{(a)}(x,y)- T_k(x)S_r^{(a)}(x,y) \\
\label{eq:part1}&\quad=
T_k(y)-T_k(x) \\
\label{eq:part2}&\qquad +\sum_{j=1}^r (1-a_j)
\Bigl(T_j(x)T_{j+k}(y)-T_{j+k}(x)T_j(y)\Bigr) \\
\label{eq:part3} &\qquad
+\sum_{j=1}^r (1-a_j)
\Bigl(T_j(x)T_{|j-k|}(y)-T_{|j-k|}(x)T_j(y)\Bigr).
\end{align}

We now split the sum~\eqref{eq:part3}.
\begin{align}
\notag &\sum_{j=1}^r (1-a_j)
\Bigl(T_j(x)T_{|j-k|}(y)-T_{|j-k|}(x)T_j(y)\Bigr)\\
\label{eq:sum<k}&\quad=
\sum_{j=1}^{k-1} (1-a_j)
\Bigl(T_j(x)T_{k-j}(y)-T_{k-j}(x)T_j(y)\Bigr)
\\
\label{eq:sum=k}&\qquad
+ (1-a_k)
\Bigl(T_k(x) - T_k(y)\Bigr) \\
\label{eq:sum>k}&\qquad
+ \sum_{j=k+1}^{r} (1-a_j)
\Bigl(T_j(x)T_{j-k}(y)-T_{j-k}(x)T_j(y)\Bigr)
\end{align}
Collecting~\eqref{eq:part1} and~\eqref{eq:sum=k} we obtain,
\[
T_k(y)-T_k(x)+(1-a_k)\bigl(T_k(x)-T_k(y)\bigr)
=
a_k\bigl(T_k(y)-T_k(x)\bigr).
\]
After the change of index $j \to k+j$ and using $a_j = 1$ for $j >r$ we obtain that~\eqref{eq:sum>k} is equal to
\[
\sum_{j=1}^{r} (1-a_{j+k})
\Bigl(T_{j+k}(x)T_j(y)-T_j(x)T_{j+k}(y)\Bigr).
\]
Combining this sum with~\eqref{eq:part2} we obtain
\[
\sum_{j=1}^r (a_{j+k}-a_j)
\Bigl(T_j(x)T_{j+k}(y)-T_{j+k}(x)T_j(y)\Bigr).
\]
Finally, the change of indices $j \to k-j$ in~\eqref{eq:sum<k} gives
\[
\sum_{j=1}^{k-1}(1-a_{k-j})
\Bigl(T_{k-j}(x)T_{j}(y)-T_{j}(x)T_{k-j}(y)\Bigr),
\]
averaging the two equivalent expressions gives
\[
\frac12\sum_{j=1}^{k-1}(a_{k-j}-a_j)
\Bigl(T_j(x)T_{k-j}(y)-T_{k-j}(x)T_j(y)\Bigr).
\]

Collecting the three contributions proves the claimed identity.
\end{proof}

We now study the image of the Chebyshev basis under the operator $\mK_r^{(a)}$. The goal is to show
$\mK_r^{(a)}(T_k)(x)/M_r^{(a)}(x) \approx T_k(x)$ for each $k\le r$, in a sense to be made precise.

\begin{proposition}\label{prop:boundCoeffsSqr}
Let  \(0\le k\le r\), then,
\begin{align*}
&\mK_r^{(a)}(T_k)(x)-M_r^{(a)}(x)\,T_k(x)\\
&\quad =
A_{k,r}^{(a)}\,T_k(x)
-\frac12\sum_{j=1}^r (a_{j+k}-a_j)^2\,T_{2j+k}(x)
-\frac14\sum_{j=1}^{k-1}(a_{k-j}-a_j)^2\,T_{|2j-k|}(x),
\end{align*}
where
\[
A_{k,r}^{(a)}
:=
-a_k^2
-\frac12\sum_{j=1}^r (a_{j+k}-a_j)^2
-\frac14\sum_{j=1}^{k-1}(a_{k-j}-a_j)^2.
\]
\end{proposition}
\begin{proof}
Multiply both sides in \cref{lem:SymSa} by \(S_r^{(a)}(x,y)\) and integrate with respect to \(d\mu(y)\). The left-hand side becomes
\[
\mK_r^{(a)}(T_k)(x)-M_r^{(a)}(x)\,T_k(x).
\]
We use
\[
\int_{-1}^1 T_m(y)\,S_r^{(a)}(x,y)\,d\mu(y)
=
\begin{cases}
(1-a_m)T_m(x), & 1\le m\le r,\\[1mm]
1, & m=0,\\[1mm]
0, & m>r.
\end{cases}
\]
For the first term on the right-hand side, we obtain
\[
a_k\Bigl((1-a_k)T_k(x)-T_k(x)\Bigr)
=
-a_k^2\,T_k(x).
\]
For the second term, each summand gives
\[
(a_{j+k}-a_j)\Bigl((1-a_{j+k})-(1-a_j)\Bigr)T_j(x)T_{j+k}(x)
=
-(a_{j+k}-a_j)^2\,T_j(x)T_{j+k}(x).
\]
For the third term, each summand gives
\[
\frac12(a_{k-j}-a_j)\Bigl((1-a_{k-j})-(1-a_j)\Bigr)T_j(x)T_{k-j}(x)
=
-\frac12(a_{k-j}-a_j)^2\,T_j(x)T_{k-j}(x).
\]

This proves
\begin{align*}
&\mK_r^{(a)}(T_k)(x)-M_r^{(a)}(x)\,T_k(x) \\
& \quad =
-a_k^2\,T_k(x)
-\sum_{j=1}^r (a_{j+k}-a_j)^2\,T_j(x)T_{j+k}(x)
-\frac12\sum_{j=1}^{k-1}(a_{k-j}-a_j)^2\,T_j(x)T_{k-j}(x).
\end{align*}
Finally, apply the product formulas
\[
2T_jT_{j+k}=T_{2j+k}+T_k,
\qquad
2T_jT_{k-j}=T_{|2j-k|}+T_k,
\]
to obtain the stated expansion.
\end{proof}

\subsection{A sum-of-squares approximation}
\label{sec:A sum-of-squares approximation}
We now have a generic construction of a positive linear operator $\mK_r^{(a)}$ that approximately maps $T_k$ to $M_r^{(a)}\,T_k$. It thus remains to divide by $M_r^{(a)}$ in a suitable way. For this we use the lemmata from \cref{sec:divide}.

Recall that, for $a \in [0,1]^r$ we let
\begin{align*}
S_r^{(a)}(x,y) &:=1+2\sum_{j=1}^r (1-a_j)\,T_j(x)T_j(y), \\
M_r^{(a)}(x)&:=\mK_r^{(a)}(1)(x)
=
1+2\sum_{j=1}^r (1-a_j)^2\,T_j(x)^2,
\end{align*}
and we extend the sequence \((a_j)\) by \(a_m:=1\) for all \(m>r\).
Then, for $N >1$ even let
\[
p_{r,N}^{(a)} := \frac1{1+2\sum_{j=1}^r (1-a_j)^2}\sum_{m=0}^{N}
\left(1-\frac{M_r^{(a)}}{1+2\sum_{j=1}^r (1-a_j)^2}
\right)^m.
\]

We now define the kernel
\begin{equation}\label{def:kerneltilde}
  \widetilde K_{r,N}^{(a)}(x,y) :=  p_{r,N}^{(a)}(x)S_r^{(a)}(x,y)^2 \quad\quad (x,y \in [-1,1]),
\end{equation}
and its associated positive linear approximation operator
  \[
\widetilde {\mathcal K}_{r,N}^{(a)}(f)(x) := \int_{[-1,1]} f(y)\widetilde K_r^{(a)}(x,y)d\mu(y).
  \]

First note that $\widetilde {\mathcal K}_{r}^{(a)}$ maps nonnegative polynomials to SOS-polynomials of degree $2r(N+1)$. To see this, one can use a quadrature rule for the Chebyshev-measure $\mu$ and the fact that $p_{r,N}^{(a)}$ is an SOS-polynomial of degree $2rN$. We now bound the error.

\begin{lemma}
\label{lemma:tildeK_approx}
Let $C:= 1+2\sum_{j=1}^r (1-a_j)^2$ and assume $c$ is a lower bound on $M_r^{(a)}(x)$ for $x \in [-1,1]$.
Then,
\begin{equation}
\|\widetilde {\mathcal K}_{r,N}^{(a)}(T_k) - T_k\|_{\cheb} \le \frac{N+1}{C}\|{\mathcal K}_{r,N}^{(a)}(T_k) -M_r^{(a)}\,T_k \|_{\cheb} + \frac{C\sqrt{2((N+2)r+1)}}{c}\left(1-\frac{c}{C}\right)^{N+1}.
\label{tildeK_approximation}
\end{equation}
\end{lemma}
\begin{proof}
By construction, for any polynomial $f$ we have
\[
\widetilde {\mathcal K}_{r,N}^{(a)}(f) = p_{r,N}^{(a)} \mK_r^{(a)}(f).
\]

We now estimate the error in the $\|\cdot\|_\cheb$-norm using the triangle inequality and sub-multiplicativity of the $\|\cdot\|_\cheb$-norm:
\begin{align*}
\|\widetilde {\mathcal K}_{r,N}^{(a)}(T_k) - T_k\|_{\cheb} &=\|\widetilde {\mathcal K}_{r,N}^{(a)}(T_k) -p_{r,N}^{(a)}M_r^{(a)}\,T_k + p_{r,N}^{(a)}M_r^{(a)}\,T_k -  T_k\|_{\cheb} \\
&\leq\|\widetilde {\mathcal K}_{r,N}^{(a)}(T_k) -p_{r,N}^{(a)}M_r^{(a)}\,T_k \|_{\cheb} + \|p_{r,N}^{(a)}M_r^{(a)}\,T_k -  T_k\|_{\cheb} \\
&=\|p_{r,N}^{(a)} \left( {\mathcal K}_{r,N}^{(a)}(T_k) -M_r^{(a)}\,T_k\right) \|_{\cheb} + \|\left(p_{r,N}^{(a)}M_r^{(a)} - 1\right)T_k\|_{\cheb} \\
&\leq \|p_{r,N}^{(a)}\|_{\cheb} \|{\mathcal K}_{r,N}^{(a)}(T_k) -M_r^{(a)}\,T_k \|_{\cheb} + \|p_{r,N}^{(a)}M_r^{(a)} - 1\|_{\cheb}.
\end{align*}
It thus remains to bound $\|p_{r,N}^{(a)}\|_\cheb$ and $\|p_{r,N}^{(a)}M_r^{(a)} - 1\|_{\cheb}$.

To upper bound $\|p_{r,N}^{(a)}\|_\cheb$ we use \cref{lem:ratio,lem:1qbound} with $q:=M_r^{(a)}$ and $C:= 1+2\sum_{j=1}^r (1-a_j)^2$ to obtain
\[
\|p_{r,N}^{(a)}\|_\cheb \leq \frac{N+1}{1+2\sum_{j=1}^r (1-a_j)^2}.
\]
To upper bound $\|p_{r,N}^{(a)}M_r^{(a)} - 1\|_{\cheb}$ we use \cref{lem:ratio} again with the same choice of $q$ and $C$ to first obtain
\begin{equation} \label{eq:sup}
\left\|\frac{1}{M_r^{(a)}} -p_{r,N}^{(a)}\right\|_{\infty} \leq \frac{1}{c}\left(1-\frac{c}{C}\right)^{N+1},
\end{equation}
where $c$ is a lower bound on $M_r^{(a)}(x)$ for $x \in [-1,1]$.   Using sub-multiplicativity of the sup-norm we obtain from \eqref{eq:sup} that
\[
\left\|1 -p_{r,N}^{(a)}M_r^{(a)}\right\|_{\infty} \leq \frac{C}{c}\left(1-\frac{c}{C}\right)^{N+1}.
\]
We now observe that $p_{r,N}^{(a)}M_r^{(a)}$ is a degree-$(N+2)r$ polynomial. For degree-$d$ polynomials, it is known that $\|p\|_\cheb \leq \sqrt{2(d+1)} \|p\|_\infty$, see \eqref{norm equivalence relations}. Hence,
\[
\left\|1 -p_{r,N}^{(a)}M_r^{(a)}\right\|_{\cheb} \leq \frac{C\sqrt{2((N+2)r+1)}}{c}\left(1-\frac{c}{C}\right)^{N+1},
\]
completing the proof.
\end{proof}

\subsection{Choosing the weights \texorpdfstring{$a_j$}{aj}} \label{sec:phi}
Let \(\phi:\mathbb{R}\to\mathbb{R}\) be continuous, with \(
\phi(x)=0\) for \(x\le 0\),
\(\phi(x)=1\) for \(x\ge 1\),
and such that \(\phi\in C^1([0,1])\).
For \(j,r > 0\), we now choose the kernel weights via
\(
a_j:=\phi\!\left(\frac{j}{r}\right)\).

\begin{lemma}\label{lem:f-bound}
For all $j,k \ge 0$,
\[ (a_{j+k}-a_j)^2 \le \frac {k^2}{r^2}\|\phi'\|_{\infty}^2, \]
in particular,
\[
\sum_{j=1}^{r}(a_{j+k}-a_j)^2
\le \frac{k^{2}}{r}\,\|\phi'\|_{\infty}^2.
\]
\end{lemma}

\begin{proof}
For each \(j\),
\[
a_{j+k}-a_j
=
\int_{j/r}^{(j+k)/r} \phi'(t)\,dt.
\]
Hence, 
\[
|a_{j+k}-a_j|
\le
\int_{j/r}^{(j+k)/r} |\phi'(t)| \,dt \leq \frac{k}{r} \|\phi'\|_\infty.
\]
Summing over \(j=1,\dots,r\), we obtain
\begin{align*}
\sum_{j=1}^{r}(a_{j+k}-a_j)^2
&\le \frac {k^2}{r}\|\phi'\|_{\infty}^2.
\end{align*}
\end{proof}

We may now formulate the bounds from the previous subsection in terms of $\phi$.
\begin{proposition}\label{prop:boundCoeffsSqr-f}
Let  \(0\le k\le r\), then,
\[\| \mK_r^{(a)}(T_k)-M_r^{(a)}\,T_k \|_\cheb \le \phi\left( \tfrac kr \right)^2 + \frac {3 k^2}{2r}\|\phi'\|_{\infty}^2.
\]
\end{proposition}
\begin{proof}
For $k=0$ we observe that $M_r^{(a)}(x):=\mK_r^{(a)}(1)(x)$. Now let $1 \leq k \leq r$.
By \cref{prop:boundCoeffsSqr}
and \cref{lem:f-bound}
\begin{align*}
 \| \mK_r^{(a)}(T_k)-M_r^{(a)}\,T_k \|_\cheb & \le
a_k^2
+\sum_{j=1}^r (a_{j+k}-a_j)^2
+\frac12\sum_{j=1}^{k-1}(a_{k-j}-a_j)^2\\
& \le \phi\left( \tfrac kr \right)^2 + \frac {3 k^2}{2r}\|\phi'\|_{\infty}^2.
\end{align*}
\end{proof}

For nonnegative kernels, we may lower bound the function $M_r^{(a)}$ in terms of $\phi$ as follows.

\begin{proposition}\label{prop:M-lower-bound}
Suppose $a \in [0,1]^r$ is such that 
\[
S_r^{(a)}(x,y)\ge 0 \text{ for all }(x,y)\in[-1,1]^2.
\]
Then, for every \(x\in[-1,1]\),
\[
M_r^{(a)}(x) \ge \frac12+\sum_{j=1}^r (1-a_j)^2.
\]
In particular,
\[
M_r^{(a)}(x)\ge \tfrac 12 - 2\|\phi'\|_\infty +  \tfrac 1{3\|\phi'\|_\infty} r.
\]
\end{proposition}

\begin{proof}
Consider the kernel composition
\[
R_r(x,y) := \int_{-1}^1 S_r^{(a)}(x,t)S_r^{(a)}(t,y)\,d\mu(t).
\]
Since \(S_r^{(a)}\) is nonnegative on \([-1,1]^2\), it follows that
\[
R_r(x,y)\ge 0 \qquad\text{for all }(x,y)\in[-1,1]^2.
\]
Using the orthogonality relations, we compute
\[
\begin{aligned}
R_r(x,y)
&= \int_{-1}^1 \left( 1+2\sum_{j=1}^r (1-a_j)T_j(x)T_j(t) \right)  \left( 1+2\sum_{k=1}^r (1-a_k)T_k(t)T_k(y) \right) d\mu(t) \\
&= 1 + 4\sum_{j,k=1}^r (1-a_j)(1-a_k)T_j(x)T_k(y) \int_{-1}^1T_j(t)T_k(t)\,d\mu(t) \\
&= 1+ 2\sum_{j=1}^r (1-a_j)^2T_j(x)T_j(y). 
\end{aligned}
\]
Now we use $2T_j(x)^2 = 1 + T_{2j}(x)  = 1 + T_j(T_2(x)) $ to write,
\[
\begin{aligned}
M_r^{(a)}(x)
&= 1+2\sum_{j=1}^r (1-a_j)^2T_j(x)^2 
= 1+ \sum_{j=1}^r (1-a_j)^2 (1+T_{j}(T_2(x))) \\
&= \frac12 \left( R_r(1,1) + R_r(T_2(x),1) \right)
\ge \frac12 R_r(1,1), 
\end{aligned}
\]
proving the first statement.
Now we prove the second statement, since \(g(t) = (1-\phi(t))^2\) in $C^1([0,1])$, by the mean value theorem we have
\[
|g(x) -  g(j/r)| \le  \|g'\|_\infty\,|\tfrac jr-x| \le \tfrac 2r\|\phi'\|_\infty \text{ for }  \tfrac {j-1}r \le x \le \tfrac{j}r.
\]
In particular,
\[
\frac{1}{r} g(j/r) \geq \int_{(j-1)/r}^{j/r} g(x) \, dx - \frac{2}{r^2}\|\phi'\|_\infty.
\]
Substituting this into the preceding bound yields
\[
M_r^{(a)}(x)
\ge \frac12+ \sum_{j=1}^r (1-a_j)^2
\ge \frac12+  r\int_0^1 (1-\phi(t))^2\,dt - 2\|\phi'\|_\infty.
\]
To finish, let $M  = \|\phi'\|_{\infty}$. For all \(t\in[0,1]\), we have \(\phi(t) \le Mt \).
Hence
\[
\int_0^1 (1-\phi(t))^2\,dt 
\ge \int_0^{1/M} (1-Mt)^2\,dt
= \frac1M\int_0^1 (1-u)^2\,du
=  \frac{1}{3M}.
\]
\end{proof}

We are now able to simplify the error bound in Lemma \ref{lemma:tildeK_approx} by formulating it in terms of $\phi$.
\begin{theorem} \label{theo:main}
    Assume  $S_r^{(a)}$ is nonnegative on $[-1,1]^2$, and consider the associated kernel $\widetilde {\mathcal K}_{r,N}^{(a)}$. %Let $c:= \frac12+\sum_{j=1}^r (1-a_j)^2$ and $C:=1+2\sum_{j=1}^r(1-a_j)^2$.
    Then
    \[
    \|\widetilde {\mathcal K}_{r,N}^{(a)}(T_k) - T_k\|_{\cheb} \leq \frac{N+1}{1 - 4\|\phi'\|_\infty +  \tfrac 2{3\|\phi'\|_\infty} r}\left(\phi\left( \tfrac kr \right)^2 + \frac {3 k^2}{2r}\|\phi'\|_{\infty}^2\right) + \sqrt{2((N+2)r+1)}\left(\frac{1}{2}\right)^{N}
    \]
    for every $0 \leq k \leq r$, provided $r > 6\|\phi'\|_{\infty}^2 - \frac{3}{2}\|\phi'\|_{\infty}$.
\end{theorem}

\begin{proof}
From \cref{prop:boundCoeffsSqr-f} we recall that for $0 \leq k \leq r$ we have
\begin{equation}
\label{bound1}
\| \mK_r^{(a)}(T_k)-M_r^{(a)}\,T_k \|_\cheb \le \phi\left( \tfrac kr \right)^2 + \frac {3 k^2}{2r}\|\phi'\|_{\infty}^2.
\end{equation}
By \cref{prop:M-lower-bound}, $c := \frac12+\sum_{j=1}^r (1-a_j)^2$ is a lower bound on $M_r^{(a)}(x)$ for $x \in [-1,1]$.
By Lemma \ref{lemma:tildeK_approx}, we have
\[
\|\widetilde {\mathcal K}_{r,N}^{(a)}(T_k) - T_k\|_{\cheb} \le \frac{N+1}{C}\|{\mathcal K}_{r,N}^{(a)}(T_k) -M_r^{(a)}\,T_k \|_{\cheb} + \frac{C\sqrt{2((N+2)r+1)}}{c}\left(1-\frac{c}{C}\right)^{N+1}.
\]
After noting that $c = \frac{1}{2}C$, and using \eqref{bound1}, this becomes
\[
\|\widetilde {\mathcal K}_{r,N}^{(a)}(T_k) - T_k\|_{\cheb} \le \frac{N+1}{2c}\left(\phi\left( \tfrac kr \right)^2 + \frac {3 k^2}{2r}\|\phi'\|_{\infty}^2\right) + \sqrt{2((N+2)r+1)}\left(\frac{1}{2}\right)^{N}.
\]
We now use
$c \ge \tfrac 12 - 2\|\phi'\|_\infty +  \tfrac 1{3\|\phi'\|_\infty}$ by Proposition \ref{prop:M-lower-bound} to obtain the required result.
Finally, the lower bound condition on $r$ in the theorem is to ensure that $r$ is sufficiently large to guarantee $\tfrac 12 - 2\|\phi'\|_\infty +  \tfrac 1{3\|\phi'\|_\infty}r > 0$.
\end{proof}

\subsection{The squared Fej\'er kernel}
\label{sec:Fejer}
Now we use Theorem \ref{theo:main} to analyze the square of the Fej\'er kernel, that corresponds to $\phi(x) = x$; see e.g.\ \cite{RevModPhys.78.275}.
We therefore consider the sequence 
\[
\hat a_j:=\frac{j}{r}\qquad (j \le r).
\]
We obtain then that $S_r^{(\hat a)}$ is the  Fej\'er kernel and $K_r^{(\hat a)}$ is its square.
It is well-known that the Fej\'er kernel is nonnegative, i.e.\ $S_r^{(\hat a)}(x,y) \ge 0$ for all $x,y \in [-1,1]$; for a proof, see e.g.\ \cite{RevModPhys.78.275}.
We again extend the sequence using $\hat a_j = 1$ for $j > r$.

It remains to choose a suitable value of $N$ in order to get an $O(\log r/r^2)$ error bound.
\begin{proposition}
\label{tildeK_final_error}
For $r \ge 5$ and  $N = \lceil 3\log_2 r \rceil$, one has
\begin{equation*}
  \|\widetilde {\mathcal K}_{r,N}^{(\hat a)}(T_k) - T_k\|_{\cheb} \le  \frac{45k^2(\log_2 r + 1) +4}{r^2} = O\left(\frac{1+k^2\log_2 r}{r^2} \right).
\end{equation*}
Moreover, for this choice of $N$, $\widetilde{\mathcal K}_{r,N}^{(a)}(T_k)$ is an SOS polynomial of degree $2r(1 +\lceil \log_2 r^3\rceil)$.
\end{proposition}
\begin{proof}
By Theorem \ref{theo:main}, we have, for $r \ge 5$, after using $\phi(x) = x$,
\begin{equation}\label{bound_from_thm2}
  \|\widetilde {\mathcal K}_{r,N}^{(\hat a)}(T_k) - T_k\|_{\cheb}  \leq \frac{2(N+1)}{  \tfrac 2{3}r  - 3}\left(\left( \frac kr \right)^2 + \frac {3 k^2}{2r}\right) + \sqrt{2((N+2)r+1)}\left(\frac{1}{2}\right)^{N}.
\end{equation}
Using $N = \lceil 3\log_2 r \rceil$, we may bound the second right-hand-side term in \eqref{bound_from_thm2} by 
\begin{eqnarray*}
\sqrt{2((N+2)r+1)}\left(\frac{1}{2}\right)^{N} &\le & \frac{\sqrt{2[(\lceil 3\log_2 r\rceil +2)r+1]}}{r^3} \\
                                                  &\le & \frac{\sqrt{6r\log_2 r +6r+2} }{r^3} < \frac{4}{r^2},\\
\end{eqnarray*}
where we have used $\lceil 3\log_2 r\rceil \le   3\log_2 r +1$, and $\log_2 r \le r$. Similarly, one 
may check by elementary calculations that the first  right-hand-side term in \eqref{bound_from_thm2} may be upper bounded by the simpler expression
\[
\frac{2(N+1)}{  \tfrac 2{3}r  - 3}\left(\left( \frac kr \right)^2 + \frac {3 k^2}{2r}\right) \le \frac{45k^2(\log_2 r + 1)}{r^2} \quad (r \ge 5).
\]
Combining the upper bounds on the two right-hand-side terms in \eqref{bound_from_thm2} leads to the required result.
\end{proof}

\subsection{Extension to the multivariate case}
\label{sec:multivariate extension}
We now construct a multivariate kernel $K^{( a)}_{r,n}(\x,\y)$ for $\x,\y \in [-1,1]^n$, by multiplying univariate kernels:
\begin{equation}
\label{multivariate kernel}
K^{(a)}_{r,n}(\x,\y) := \prod_{i \in [n]} K^{(a)}_r(x_i,y_i),
\end{equation}
where $[n] = \{1,\ldots,n\}$.
Proceeding as before, we define a convolution operator,
\[
\widetilde{\mathcal{K}}^{(a)}_{r,n,N}f(\x) := \prod_{i \in [n]} p_{r,N}^{(a)}(x_i)\int_{[-1,1]^n} f(y) K^{(a)}_{r,n}(\x,\y)d\mu(\y),
\]
where $d\mu(\y) = d\mu(y_1)\ldots d\mu(y_n)$.
The operator  $\widetilde{\mathcal{K}}^{(a)}_{r,n,N}$, approximately preserves multivariate Chebyshev polynomials $T_\alpha(\x) = \prod_{i \in [n]} T_{\alpha_i}(x_i)$ in the following sense.
\begin{lemma} \label{thrm:approxIDmulti}
    Let $\alpha \in (\N_0)^n$ be such that $\alpha_i \leq d$ for all $i \in [n]$.
    Also assume that, for given $\eps >0$, one has $\|\widetilde{\mathcal K}^{(a)}_{r,N}(T_{\alpha_i})- T_{\alpha_i}\|_\cheb \le \eps$ for all $i \in [n]$.
     Then we have
    \[
    \|\widetilde{\mathcal K}^{(a)}_{r,n,N} T_\alpha - T_\alpha\|_\cheb \leq \eps \sum_{i=0}^{n-1} (1+\eps)^i.
    \]
    Moreover, if $\eps \leq 1/n$, then $\|\widetilde{\mathcal K}^{(a)}_{r,n} T_\alpha - T_\alpha\|_\cheb \leq e \cdot \eps \cdot n$, where $e \approx 2.71828$ is the Euler number.
\end{lemma}
\begin{proof}
The result follows immediately from \cite[Lemma 4.2]{GdKV26}, which is a general result for the product of univariate kernels.
\end{proof}

By using Proposition \ref{tildeK_final_error}, we have the following result for the multivariate squared Fej\'er kernel.

\begin{proposition}
\label{tildeK_final_error multivariate}
 Let $\alpha \in (\N_0)^n$ be such that $\alpha_i \leq d$ for all $i \in [n]$. For $N = \lceil 3\log_2 r \rceil$, and $r \ge 45nd^2$, one has
\begin{equation*}
   \|\widetilde{\mathcal K}^{(\hat a)}_{r,n,N} T_\alpha - T_\alpha\|_\cheb \le \frac{e\cdot n (45d^2(\log_2 r + 1) +4)}{r^2}.
\end{equation*}
Moreover, for this choice of $N$, $\widetilde{\mathcal K}_{r,n,N}^{(\hat a)}(T_\alpha)$ is an SOS polynomial of degree  $2rn\lceil 1+3\log_2 r\rceil$.
\end{proposition}
\begin{proof}
The result follows immediately from Lemma \ref{thrm:approxIDmulti} and Proposition \ref{tildeK_final_error}, after ensuring that the $\eps \le 1/n$ condition of Lemma \ref{thrm:approxIDmulti} holds. To this end, we need to ensure that $r$ is sufficiently large, so that 
\begin{equation}\label{one over n condition}
\eps =   \frac{45d^2(\log_2 r + 1) +4}{r^2} \le \frac{1}{n},
\end{equation}
where the  expression for $\eps$ is from Proposition \ref{tildeK_final_error} with $k=d$, and holds when $r \ge 5$.
Since $\log_2 r + 1 < \frac{2}{3}r$ for $r \ge 5$, one may easily show that \eqref{one over n condition} holds, e.g., if $r \ge 45nd^2$.
\end{proof}

\section{Main result and its consequences}
\label{sec:Main result and its consequences}
In this section we first state and prove our main result on sum-of-squares approximation of nonnegative polynomials on $[-1,1]^n$. Then we explore the implications for two different hierarchies for approximating the minimum of a polynomial on the hypercube, both due to Lasserre.
\subsection{SOS approximations of nonnegative polynomials on the hypercube}
\label{sec:SOS approximations of nonnegative polynomials on the hypercube}
  \begin{theorem}
  \label{thm:main}
Assume $f(\x) = \sum_{\alpha \in \N^n_d} f_\alpha T_\alpha (\x)$  is nonnegative on $[-1,1]^n$.
There exist absolute, positive constants $c_1\le 70{,}458$ and $c_2 \le 540$, so that,
for any integer $r$ with $n|r$ and $r/\log_2 r \ge c_2n^2d^2$, there exists
a $p \in \Sigma[\x]_{r}$ such that
\[
\|p - f \|_\cheb
\le
c_1n^3d^{2}\|f\|_\cheb \cdot{\frac{\log^3_2(r)}{r^2}}.
\]
\end{theorem}
\begin{proof}
By Proposition \ref{tildeK_final_error multivariate}, for 
$N = \lceil 3\log_2 r \rceil$, and $r \ge 45nd^2$, one has
\begin{eqnarray*}
\|\widetilde{\mathcal K}^{(\hat a)}_{r,n,N} f - f\|_\cheb & = & \left\|\sum_{\alpha \in \N^n_d} f_\alpha\left(\widetilde{\mathcal K}^{(\hat a)}_{r,n,N} T_\alpha - T_\alpha\right)\right\|_\cheb \\ 
&\le & \sum_{\alpha \in \N^n_d} |f_\alpha|\left\|\widetilde{\mathcal K}^{(\hat a)}_{r,n,N} T_\alpha - T_\alpha\right\|_\cheb \\ 
&\le & \|f\|_\cheb \cdot \frac{e\cdot n (45d^2(\log_2 r + 1) +4)}{r^2}.
\end{eqnarray*}
Setting $p:= \widetilde{\mathcal K}^{(\hat a)}_{r,n,N} f$ it follows that $p \in \Sigma[\x]_\rho$ with $\rho := 3rn\lceil 1+ 3 \log_2 r\rceil$.
We may now write the error bound in terms of $\rho$ as follows:
\begin{eqnarray*}
  \|f\|_\cheb \cdot\frac{e\cdot n (45d^2(\log_2 r + 1) +4)}{r^2} &=& \|f\|_\cheb \cdot \frac{3^2n^2\lceil 1+ 3 \log_2 r\rceil^2 \cdot e\cdot n (45d^2(\log_2 r + 1) +4)}{\rho^2} \\
 &\le & \|f\|_\cheb \cdot\frac{9\cdot 45\cdot e \cdot n^3d^2\lceil 1+ 3 \log_2 r\rceil^3}{\rho^2} \\
  &\le & \|f\|_\cheb \cdot\frac{9\cdot 45\cdot e \cdot n^3d^2\lceil 1+ 3 \log_2 \rho\rceil^3}{\rho^2} \\
   &\le & \|f\|_\cheb \cdot\frac{9\cdot 45\cdot e \cdot n^3d^2\lceil 4\log_2 \rho\rceil^3}{\rho^2} \\
   &\le & c_1\cdot \|f\|_\cheb \cdot n^3d^{2} \cdot{\frac{\log^3_2(\rho)}{\rho^2}},
\end{eqnarray*}
where $c_1\le 70{,}458$ is an absolute constant.
It remains to formulate the lower bound  $r \ge 45nd^2$ from Proposition \ref{tildeK_final_error multivariate} in terms of $\rho$.
By the definition of $\rho$, the required lower bound will hold if
\[
\rho \ge  45\cdot 3n^2d^2\lceil 1+ 3 \log_2 \rho\rceil,
\]
which in turn will hold if 
\[
\frac{\rho}{\log_2 \rho} \ge 540n^2d^2.
\]
\end{proof}

\begin{remark}
The error bound in Theorem \ref{thm:main}
is tight in terms of the dependence on $r$ up to a poly-logarithmic factor, due to the following result.

\begin{theorem}[Theorem 5.3 in \cite{GdKV26}]
Consider the univariate polynomial $p(x) := 1-x^2$. There exists an absolute constant $C>0$ such that
 $\|p-q\|_\cheb \ge \frac{C}{r^2}$ for any $r \in 2\N$ and $q \in \Sigma[x]_r$.
\end{theorem}
\end{remark}

\subsection{Improved rate of convergence for the Lasserre hierarchy of lower bounds}
\label{sec:Improved rate of convergence for the Lasserre hierarchy of lower bounds}
Consider the problem
\begin{equation}
\label{prob:origin}
f_\min := \min_{\x \in [-1,1]^n} f(\x),
\end{equation}
with $f$ a degree $d$ polynomial, and define $f_\max$ analogously.
We describe $[-1,1]^n$ via the constraints
\[
  g_i(\x) := 1-x_i^2 \ge 0 \quad(i \in \{1,\ldots,n\}).
\]
The associated {truncated  quadratic module} of order $r$, generated by $\mathbf{g} = (g_1,\ldots,g_n)$, is defined as
\[
\mathcal{Q}(\mathbf{g})_r = \Sigma[\x]_r + \sum_{i=1}^n g_i \Sigma[\x]_{r - 2}.
\]
For this description of the hypercube, the Lasserre hierarchy \cite{doi:10.1137/S1052623400366802} for problem \eqref{prob:origin} is defined by:
\begin{equation}
\label{eq:lasserre hierarchy}
f_{(r)} := \sup \left\{ t \; : \; f - t \in \mathcal{Q}(\mathbf{g})_r\right\} \quad r \in \N.
 \end{equation}
One has $f_{(r)} \le f_{\min}$ for all $r$ and $f_{(r)} \rightarrow f_\min$ as $r \rightarrow \infty$.
The best result on the rate of convergence is the following $O(1/r)$ result due to Baldi and Slot \cite{doi:10.1137/23M1555430}.

\begin{theorem}[cf.\  Corollary 15 in \cite{doi:10.1137/23M1555430}]
\label{cor:Baldi_Slot}
Consider problem \eqref{prob:origin}, and let $r \in \N$ satisfy $n|r$ as well as
\[
r\ge  4 \Chebyconst \cdot d^2(n\log n) + n  + 2n\sqrt{\Chebyconst C(n,d)},
\]
where $\Chebyconst \in [1,e^5]$ is an absolute constant with $e \approx 2.71828$ being Euler's number, and $C(n, d)$ is a constant that only depends on $(n,d)$, and satisfies:
\[
2\pi^2d^2n \binom{n+d}{d} \le C(n, d)\le 2\pi^2d^2n\cdot \min \left\{ 2^{n/2}(d + 1)^n, 2^{d/2}(n + 1)^d\right\}.
\]
Then
\[
% \label{eq:lasserre hierarchy}
f_{(r)} - f_\min \le \frac{4 \Chebyconst \cdot d^2(n\log n)}{r}(f_\max - f_\min).
\]
\end{theorem}

Our main result in Theorem \ref{thm:main} allows us to improve this result by showing a rate of $O(\log^3  r/r^2)$. To this end, we need the following result from
Gribling et al. \cite[Corollary 10]{GdKV26}.
\begin{theorem}
\label{thm:lasserre hierarchy error norm one}
   Consider problem \eqref{prob:origin} and its associated Lasserre hierarchy \eqref{eq:lasserre hierarchy}.
One has, for $r \ge \mathrm{deg}(f)$,
\[
f_\min - f_{(2r)} \le \min_{q \in Q(1-x_1^2,1-x_2^2,\ldots, 1-x_n^2)_{2r}} \left\| f - f_\min - q\right\|_\cheb
\le \min_{q \in \Sigma[\x]_{2r}} \left\| f - f_\min - q\right\|_\cheb.
\]
\end{theorem}
Consequently, we may use our result in Theorem \ref{thm:main} to derive an improved bound on the rate of convergence of the Lasserre hierarchy \eqref{eq:lasserre hierarchy}.

\begin{theorem}
\label{thm:new error bound lasserre lower bounds}
   Consider problem \eqref{prob:origin} and its associated Lasserre hierarchy \eqref{eq:lasserre hierarchy}.
There exists an absolute, positive constants $c_1\le 70{,}458$ and $c_2 \le 540$, so that,
for any integer $r$ with $n|r$ and $r/\log_2 r \ge c_2n^2d^2$,
\[
f_\min - f_{(2r)} \le
c_1n^3d^{2}\|f\|_\cheb \cdot{\frac{\log^3_2(r)}{r^2}}.
\]
\end{theorem}
\begin{proof}
The proof follows immediately from Theorem \ref {thm:lasserre hierarchy error norm one} and Theorem \ref{thm:main}.
\end{proof}

\begin{remark}
    Theorem \ref{thm:new error bound lasserre lower bounds} 
    improves 
    the result of Baldi and Slot in Theorem \ref{cor:Baldi_Slot}  both in terms of the rate of convergence, and the lower bound on $r$. However, the bound by Baldi and Slot involves $f_\max - f_\min$, where we may assume w.l.o.g.\ that 
\[
\|f\|_\infty \le f_\max - f_\min \le 2\|f\|_\infty,
\]
by replacing $f$ by $f - f(0)$.
Our new bound in Theorem \ref{thm:new error bound lasserre lower bounds} is in terms of $\|f\|_\cheb$, which can be much larger than $\|f\|_\infty$; see \eqref{norm equivalence relations}.
\end{remark}

\subsection{A new convergence rate proof for the Lasserre spectral hierarchy of upper bounds}
\label{sec:A new convergence rate proof for the Lasserre spectral hierarchy of upper bounds}
Lasserre \cite{Las11} also proposed the following hierarchy of upper bounds on $f_\min$ (see \eqref{prob:origin}):
\begin{eqnarray}\label{fundr}
{f}^{(r)}:=\inf_{\sigma\in\Sigma[\y]_{2r}}\int_{[-1,1]^n}\sigma(\y)f(\y) d\mu(\y) \ \ \mbox{s.t. $\int_{[-1,1]^n}\sigma(\y)d\mu(\y)=1$,}
\end{eqnarray}
where $d\mu(\y) := d\mu(y_1)\ldots d\mu(y_n)$.
This may be interpreted as finding the probability distribution on $[-1,1]^n$, having density $\sigma\in\Sigma[\y]_{2r}$ with respect to the multivariate Chebyshev measure, that minimizes the expected function value of $f$.
One has $f^{(r)} \ge f_{\min}$ for all $r$ and $f^{(r)} \rightarrow f_\min$ as $r \rightarrow \infty$.

Problem \eqref{fundr} has a reformulation as a smallest eigenvalue problem. Indeed, ${f}^{(r)}$ is the smallest eigenvalue
of the symmetrix matrix $A$, with rows and columns indexed by $\N^n_r = \left\{\alpha \in (\N_0)^n \; : \; \sum_{i=1}^n \alpha_i \le r\right\}$, defined by:
\begin{equation}
\label{matrices A and B}
A_{\alpha, \beta} = \sum_{\delta \in \N^n_d} f_{\delta} \int_{[-1,1]^n} \widehat T_{\alpha}(\y) \widehat T_{\beta}(\y) \widehat T_{\delta}(\y) d\mu(\y) \quad \alpha, \beta \in \N^n_r.
\end{equation}
Moreover, if $\gamma$ denotes an eigenvector with unit Euclidean norm, corresponding to the minimum eigenvalue of $A$, then the optimal density function is given by
\begin{equation}\label{optimal_density}
  \sigma(\y) = \left( \sum_{\alpha \in \N^n_r} \gamma_\alpha \widehat T_\alpha(\y)\right)^2.
\end{equation}
De Klerk and Laurent \cite{DKL MOR2} proved that $f^{(r)} - f_{\min} = O(1/r^2)$ and that this dependence on $r$ is tight. (Here, the big-O notation suppresses dependence on all parameters except $r$.)
We obtain a new, constructive proof of this result from our kernel construction, as follows. This connection between the polynomial kernel method (for suitable kernels) and Lasserre's hierarchy of upper bounds is known, see for instance \cite[Chapter~6.2]{Slot_Thesis}.

\begin{proposition}
\label{prop:constructive proof}
Assume $f \in \R[\y]_d$ and $r \in \N$ such that $r \ge 3d\sqrt{n} +1$. 
One then has
\[
f^{(rn)} - f_{\min} \le {\frac{9\cdot e \cdot nd^{2}\|f\|_\cheb}{r(r-1)}},
\]
where $e \approx 2.71828$ is the Euler number.
\end{proposition}
\begin{proof}
Denote the minimizer of $f$ on $[-1,1]^n$ by $\x^*$, and define the density function
\[
\sigma(\y) := K^{(\hat a)}_{r,n}(\x^*,\y)\big/\prod_{i \in [n]} M^{(\hat a)}_r(x^*_i) \quad (\y \in [-1,1]^n),
\]
where $K^{(\hat a)}_{r,n}$ is the product of univariate kernels from \eqref{multivariate kernel} with $a = \hat a$.
By construction, $\sigma \in \Sigma_{2rn}[\y]$, and 
\[
\int_{[-1,1]^n}\sigma(\y)d\mu(\y)=1.
\]
Proposition \ref{prop:boundCoeffsSqr-f}
shows that
\[\| \mK_r^{(\hat a)}(T_k)-M_r^{(\hat a)}\,T_k \|_\cheb \le \left( \frac kr \right)^2 + \frac {3 k^2}{2r} < \frac {3 k^2}{r}.
\]
Moreover, by Proposition \ref{prop:M-lower-bound},
\[
M^{(\hat a)}_r(x^*_i)\ge \frac12 + \sum_{j=1}^r \left(\frac{j}{r}\right)^2 \ge \frac{1}{3}(r-1).
\]
Combining, we obtain, for all $i \in [n]$,
\[
|\mK_r^{(\hat a)}(T_k)(x_i^*)/M^{(\hat a)}_r(x^*_i) - T_k(x_i^*)| \le \frac{9k^2}{r(r-1)}.
\]
We may now construct a similar bound in the multivariate case, by proceeding as we did in Section \ref{sec:multivariate extension}, and using \cite[Lemma 4.2]{GdKV26}. In particular,
for given $ \alpha \in \N^n_d$, we have 
\begin{equation}\label{bound2}
  \left|\int_{[-1,1]^n}T_\alpha(\y) \sigma(\y)d\mu(\y) - T_\alpha(\x^*)\right| \le \frac{9n\cdot e \cdot d^2}{r(r-1)},
\end{equation}
provided that $r$ is large enough to guarantee $\frac{9d^2}{r(r-1)} \le 1/n$, e.g.\ if $r \ge 3d\sqrt{n} +1$.
Finally, if $f(\y) = \sum_{\alpha \in \N^n_d} f_\alpha T_\alpha (\y)$, then 
\begin{eqnarray*}
\left|\int_{[-1,1]^n}f(\y) \sigma(\y)d\mu(\y)-f(\x^*)\right| &=& \left|\sum_{\alpha \in \N^n_d} f_\alpha \int_{[-1,1]^n} T_\alpha(\y) \sigma(\y)d\mu(\y)-f(\x^*)\right| \\
&=& 
\left|\sum_{\alpha \in \N^n_d} f_\alpha \left(\int_{[-1,1]^n} T_\alpha(\y) \sigma(\y)d\mu(\y) -T_\alpha(\x^*)\right)\right| \\ 
&\le& 
\sum_{\alpha \in \N^n_d} \left|f_\alpha\right|  \left|\int_{[-1,1]^n} T_\alpha(\y)) \sigma(\y)d\mu(\y) -T_\alpha(\x^*)\right| \\ 
&\le & \|f\|_\cheb \cdot \frac{9n\cdot e \cdot d^2}{r(r-1)},
\end{eqnarray*}
 where the last inequality follows from \eqref{bound2}.
\end{proof}

\begin{remark}
Proposition \ref{prop:constructive proof} provides a constructive proof of the result by De Klerk and Laurent \cite[Theorem 6]{DKL MOR2}. The proof in 
\cite[Theorem 4.1]{DKL MOR2} is not constructive, and relies on known bounds for extremal roots of Jacobi polynomials. Our construction may be seen as an approximation of the eigenvector $\gamma$ in \eqref{optimal_density}, namely
\[
\gamma_\alpha \approx \frac{\prod_{i \in [n]} (1-\hat a_{\alpha_i})\widehat T_{\alpha_i}(x^*_i) }{\sqrt{\sum_{\beta \in \N^n_r} \prod_{i \in [n]} (1-\hat a_{\beta_i})^2\widehat T^2_{\beta_i}(x^*_i)}} \quad (\alpha \in \N^n_r).
\]
Thus, our approximation of $\gamma$ depends only on $f$ through its minimizer $\x^*$.
In fact, our approach to study squared kernels in this paper was motivated by the expression \eqref{optimal_density}.

In addition, Proposition \ref{prop:constructive proof} makes explicit how the error bound depends on the parameters $\|f\|_\cheb$, $n$, and $d$.
In \cite[Theorem 4.1]{DKL MOR2}, the result was simply stated as a $O(1/r^2)$ result without exploring this dependence.
\end{remark}

\section{Concluding remarks}
\label{sec:conclusion}
Our main result in Theorem \ref{thm:main} does not imply that the cone of sum-of-squares is dense in the cone of nonnegative polynomials on $[-1,1]^n$  in the $\|\cdot\|_\mon$ norm, since our error bound is formulated in terms of the $\|\cdot \|_\cheb$ norm; this is due to the equivalence relations between these norms in \eqref{norm equivalence relations2}.
In view of the known results by Berg \cite{Berg} and Lasserre \cite{doi:10.1137/04061413X}, as discussed in the introduction, it would be interesting to obtain an analogous result to Theorem \ref{thm:main} for the $\|\cdot\|_\mon$ norm.

As mentioned already, our $O(\log^3 r/r^2)$ error bound in Theorem \ref{thm:main} is tight up to the poly-logarithmic factor $\log^3 r$. It remains to be seen whether this factor may be avoided by a more careful analysis.

The error bounds for the Lasserre hierarchy of upper bounds were discussed in Section \ref{sec:A new convergence rate proof for the Lasserre spectral hierarchy of upper bounds}. Slot and Laurent \cite{Slot_Laurent_Lasserre_ub} extended these bounds to more general convex bodies than $[-1,1]^n$. Thus, it is a natural question whether our new kernel construction also extends to more convex bodies.

For the Lasserre hierarchy of lower bounds, we have obtained an $O(\log^3 r /r^2)$ error bound (see Section \ref{sec:Improved rate of convergence for the Lasserre hierarchy of lower bounds}). Similar $O(1/r^2)$ error bounds have been obtained for the sphere in \cite{10.1007/s10107-020-01537-7} and for the Euclidean ball \cite[Theorem 3]{Slot 2022}, see \cite{Laurent_Slot_Survey} for an overview of recent progress. Our improved bounds for the hypercube imply improved error bounds for the same hierarchy on general semialgebraic sets $S = \{x \in \R^n: g_i(x) \ge 0,\, i=1,\dots,m\}$ satisfying the Archimedean condition. As shown in \cite{genSem26}, from our improved bound on the hypercube, an $O(\log^3 r /r^{1/L_g})$ bound is obtained in this case, where $L_g$ is the \L{}ojasiewicz constant for $g$.

It is expected, but not known, whether the Lasserre hierarchy converges at the rate $O(1/r^2)$ for the hypercube, ball, and sphere. For the hypercube, the worst-known example, due to Baldi and Slot \cite{doi:10.1137/23M1555430}, has a convergence rate of $O(1/r^8)$. Finding an example where the Lasserre hierarchy converges at the conjectured rate would be desirable.

\bibliographystyle{plain}

\end{document}